\documentclass[reqno]{amsart}
\usepackage{amsmath,amsxtra,amssymb,latexsym, amscd,amsthm,amsfonts}
\usepackage{enumerate}
\usepackage{float}

\newtheorem{theorem}{Theorem}[section]

\newtheorem{corollary}[theorem]{Corollary}
\theoremstyle{definition}
\newtheorem{definition}[theorem]{Definition}

\numberwithin{equation}{section}

\def\cal{\mathcal}

\def\B{\cal B}

\def\r{\mathbb R}

\def\z{\mathbb Z}
\def\p{\varphi}

\begin{document}
\title[Extending a theorem of Datko]{Extending a theorem of Datko for evolutionary family} 
         
\thanks{This research is funded by Vietnam National Foundation for Science and Technology Development (NAFOSTED) under grant number 101.02-2019.01 and the Simons Foundation Grant Targeted for Institute of Mathematics, Vietnam Academy of Science and Technology.}



\author{Trinh Viet Duoc}
\address{Trinh Viet Duoc,
Faculty of Mathematics, Mechanics, and Informatics\\ University of Science, Vietnam National University\\
334 Nguyen Trai, Hanoi, Vietnam; and 
 Thang Long Institute of Mathematics and Applied Sciences, Nghiem Xuan Yem, Hanoi, Vietnam}
\email{tvduoc@gmail.com, duoctv@vnu.edu.vn}

%\dedicatory{\emph {To the memory of Professor Nguy$\tilde{\text{\^{e}}}$n Th$\acute{\text{\^{e}}}$ Ho\`{a}n}}

\date{}          

\begin{abstract}
In this note, we extend a Datko's result in the paper \cite[1972]{Dat}. In particular, the exponential stability of an evolutionary family is characterized by its pointwise trajectories in which the norm mapping of each pointwise trajectory lies in a Banach function space.
\end{abstract} 

\maketitle

\section{Introduction} 
In this note, $X$ is a real or complex Banach space with a norm $\|\cdot\|$ and $I$ is either $\r$ or $\r_+$.
The concept of an evolutionary family arises naturally from the well-posed theory of non-autonomous abstract differential equations on $\r$ or $\r_+$, the reader can refer Engel and Nagel \cite{NaE}, Chicone and Latushkin \cite{ChY}, Pazy \cite{Paz}, Daleckii and Krein \cite{DaK} on this theme.
\begin{definition}
A family of bounded linear operators $(U(t, s))_{t\ge s}$ on a Banach space $X$ is an
 (strongly continuous, exponentially bounded) evolutionary family on $I$ (that means $t,s\in I$) if
\begin{enumerate}[(i)]
\item $U(t,t)=Id$ and $U(t,r)U(r,s)=U(t,s)$ for all $t\ge r\ge s$ and $t,r,s\in I$.
\item The map $(t,s)\mapsto U(t,s)x$ is continuous for every $x\in X$.
\item There are constans $K, c>0$ such that $\|U(t,s)x\| \le Ke^{c(t-s)}\|x\|$ for all $t\ge s$ and $x\in X$.
\end{enumerate}
\end{definition}  
Among the many interesting stable types of evolutionary family, here we interest to exponential stability because it provides good information that exponential decay of  evolutionary family.
\begin{definition}
An evolutionary family $(U(t, s))_{t\ge s}$ on $I$ will be called {\it exponentially stable}  if there exist constants $K_1, \alpha>0$ such that $\|U(t,s)x\|\le K_1 e^{-\alpha(t-s)}\|x\|$ for all $x\in X$ and $t\ge s$ with $t,s\in I$.
\end{definition}
By transitive and exponentially bounded properties, the exponential stability and the uniformly asymptotic stability of an evolutionary family are equivalent together, see \cite[Lemma 1]{Dat}. To study the exponential stability of an evolutionary family, there are two basic approaches. One is based on Perron's method, for instance \cite{MinhRaSch}, and the other is based on Datko's result, e.g. \cite{VanNee}. Because we extend Datko's result so it is recalled here for easy tracking and comparison.

\medskip
\noindent
{\bf Datko Theorem.}
{\it An evolutionary family $(U(t, s))_{t\ge s}$ is exponentially stable  on $\r_+$ if and only if for each $x\in X$ there exists a constant $M(x)>0$ such that  
$$ \int_{t_0}^{\infty} \|U(t,t_0)x\|^2 dt \le M(x) \;\; \text{  for all } \;\; t_0\ge 0.$$}

In case the evolutionary family $(U(t, s))_{t\ge s\ge 0}$ is a strongly continuous semigroup that means $U(t,s)=U(t-s,0)$ for $t\ge s\ge 0$ or $U(t,s)=T(t-s)$ with $(T(t))_{t\ge 0}$ is a strongly continuous semigroup, this result was extended by Pazy in \cite{Paz} for the Lebesgue spaces $L^p(\r_+)$ with $p\in [1,\infty)$  and Neerven in \cite{VanNee}. The Neerven's result covers a class of function spaces very wide. However, it only holds on for strongly continuous semigroups and  Banach spaces over the complex number field. In addition, another limitation in Neerven's result only offers to sufficient condition for exponential stability of strongly continuous semigroups.

Follow the second approach in studying the exponential stability of an evolutionary family, our purpose is generalization of Datko Theorem for evolutionary family on the line or the half-line and a class of function spaces which is wide enough. To accomplish this task, we introduce the concept of Banach function space in Subsection \ref{s1}. Although the class of Banach function spaces in this note is smaller than the class of function spaces in \cite{VanNee} but it is still very large, most of known function spaces belong to class of Banach function spaces. Using the concept of Banach function space, we get the expected results which are stated in Subsection \ref{mr} and are proved in Section \ref{s3}. Our results hold on for both Banach spaces over complex and real field.
 
\section{Banach function spaces and main results}
\subsection{Banach function spaces}\label{s1}
In the monograph book \cite[Chapter 2]{MasSch}, Massera and Sch\"{a}ffer introduced several classes of function spaces that play a fundamental role to study differential equations. Based on the concepts of function spaces were given by Massera and Sch\"{a}ffer, we have collected some basic properties to give the notion of Banach function space. 
  
\begin{definition}\label{fs}
Let $\B$ be the Borel algebra and let $\mu$ be the Lebesgue measure on $\r$. A vector space $E$ of real-valued measurable functions on $\r$  is called  a Banach function space  if 
\begin{enumerate}[i.]
\item $(E,\|\cdot\|_{E})$ is a Banach space and $\|\cdot\|_{E}$ guarantees the property: if $\varphi_2\in E$ and $\varphi_1$ is  real-valued measurable function such that $|\p_1(t)| \le |\varphi_2(t)|$ almost everywhere on $\r$ then $\p_1\in E$ and $\|\p_1\|_{E}\le \|\varphi_2\|_{E}$;
\item the characteristic function $\chi_{[a,b]}$ belongs to $E$ for all $[a,b]\subset \r$ and  $\inf_{t\in \r}\|\chi_{[t,t+1]}\|_{E}>0$;
 \item there is a constant $M\ge 1$ such that
\begin{equation}\label{eb}
\frac{1}{b-a}\int_a^b|\varphi(t)|dt\le 
\frac{M \|\varphi\|_{E}}{\|\chi_{[a,b]}\|_{E}}\, \hbox{ for all }
\varphi\in E \,\hbox{ and } [a,b]\subset \r;
\end{equation}
\item the function $\Lambda_1\varphi$ defined by
  $(\Lambda_1\varphi)(t)= \int_t^{t+1}\varphi(\tau)d\tau$ belongs to $E$ for each $\varphi\in E$;
 \item $E$ is $T_\tau$-invariant and there exists a constant $N>0$ such that $\|T_\tau\| \le N$ for all $\tau \in \r$, where $T_\tau$ 
 is shift operator on $E$ defined by $(T_\tau\varphi)(t)=\varphi(t+\tau) \hbox{ for } t\in \r$.
\end{enumerate}
\end{definition}

Denote by $L_{1, \rm{loc}}(\r)$ space of real-valued locally integrable functions on $\r$. A family of seminorms defining the topology of $L_{1, \rm{loc}}(\r)$ is given by $\{p_n:  n\in \z \text{ and } p_n(\p)=\int_{n}^{n+1}|\p(t)|dt\}$. Then, $L_{1, \rm{loc}}(\r)$ is a Fr\'{e}chet space. Therefore, by \eqref{eb} then $E\hookrightarrow L_{1, \rm{loc}}(\r)$.

By direct inspection, it easily sees that class of Banach function spaces includes the Lebesgue spaces $L^p(\r)$ with $p\in [1,\infty]$, the Lorentz spaces $L^{p,q}(\r)$ with $p, q \in [1,\infty]$, the Orlicz spaces, etc, and the space 
$$ \mathbf{M}(\mathbb{R})=\Big\{\p \in L_{1,\rm{loc}}(\mathbb{R}) : 
\sup_{t\in \r}\int_{t}^{t+1}|\p(\tau)|d\tau <\infty\Big\}$$
with the norm $\|\p\|_{\mathbf{M}}:=\sup_{t\in \r}\int_{t}^{t+1}|\p(\tau)|d\tau$. 
By \eqref{eb} and ii. in Definition \ref{fs},
\begin{equation}\label{ebm}
\|\p\|_{\mathbf{M}} \le \frac{M}{\inf_{t\in \r}\|\chi_{[t,t+1]}\|_{E}}\|\varphi\|_{E}
\end{equation}
for all $\p \in E$. Thus, $E\hookrightarrow  \mathbf{M}(\mathbb{R})$.

Throughout this note function $\chi_{D(\p)}\p$ has the domain $\r$ and defines as follows $(\chi_{D(\p)}\p)(t)=\p(t)$ if $t\in D(\p)$ and $(\chi_{D(\p)}\p)(t)=0$ if otherwise, where $D(\p) \subset \r$ is the domain of the function $\p$.

\subsection{Main results}\label{mr}
Let $(U(t,s))_{t\ge s}$ be an evolutionary family on $I$. For  $t_0\in I$ and $x\in X$, put
$$ g_{t_0,x}(t)=\|U(t,t_0)x\| \;\text{ for }\; t\ge t_0.$$
Because the evolutionary family is strongly continuous so $\chi_{[t_0,\infty)}g_{t_0,x}$ is measurable function.
Through the Banach function space, the exponential stability of the evolutionary family will now be characterized by its pointwise trajectories. The first we put forward necessary condition.
\begin{theorem}\label{t1}
Let $E$ be any Banach function space. If $(U(t,s))_{t\ge s}$ is exponentially stable on $I$ then for each $x\in X$ function $\chi_{[t_0,\infty)}g_{t_0,x}$ belongs to $E$ and there is a constant $M(x)>0$ such that 
$$ \|\chi_{[t_0,\infty)}g_{t_0,x}\|_E \le M(x)$$
for all $t_0\in I$.
\end{theorem}

In sufficient condition, we need to add a constraint of Banach function space.
\begin{theorem}\label{t2}
Let $E$ be a Banach function space such that $\displaystyle\lim_{t\to+\infty}\|\chi_{[0,t]}\|_E=\infty$. If for each $x\in X$ function $\chi_{[t_0,\infty)}g_{t_0,x}$ belongs to $E$ and there exists a constant $M(x)>0$ such that 
$$ \|\chi_{[t_0,\infty)}g_{t_0,x}\|_E \le M(x)$$
for all $t_0\in I$ then the evolutionary family $(U(t,s))_{t\ge s}$ is exponential stability on $I$.
\end{theorem}

The constraint $\displaystyle\lim_{t\to+\infty}\|\chi_{[0,t]}\|_E=\infty$ is really necessary and not to be missed, moreover it also appears naturally in the proof of the theorem. Easily see that the Lebesgue spaces $L^p(\r)$ with $p\in [1,\infty)$ and the Lorentz spaces $L^{p,q}(\r)$ with $p\in [1,\infty), q\in [1,\infty]$ satisfy the constraint above. So, Theorem \ref{t1} and Theorem \ref{t2} are an extension for Datko's result in the paper \cite[Theorem 1]{Dat}.

The following corollaries are a minor weakening of Theorem \ref{t2} for special evolutionary families.
\begin{corollary}\label{co1}
Let $E$ be a Banach function space such that $\displaystyle\lim_{t\to+\infty}\|\chi_{[0,t]}\|_E=\infty$ and $(T(t))_{t\ge 0}$ be a strongly continuous semigroup. If  function $\chi_{[0,\infty)}g_{0,x}$ belongs to $E$ for each $x\in X$ then  $(T(t))_{t\ge 0}$ is exponentially stable, where $g_{0,x}=\|T(t)x\|$ for $t\ge 0$.
\end{corollary}
This corollary is more generalized than Pazy's result in \cite[Chapter 4, Theorem 4.1]{Paz}. The next is periodic evolutionary family.
\begin{corollary}\label{co2}
Let $E$ be a Banach function space such that $\displaystyle\lim_{t\to+\infty}\|\chi_{[0,t]}\|_E=\infty$ and  $(U(t,s))_{t\ge s}$ be a periodic evolutionary family with period $T>0$. If  function $\chi_{[T,\infty)}g_{T,x}$ belongs to $E$ for each $x\in X$  then the evolutionary family $(U(t,s))_{t\ge s}$ is exponential stability on $I$.
\end{corollary}
Note that  evolutionary family $(U(t,s))_{t\ge s}$ is periodic with period $T>0$ if $U(t+T,s+T)=U(t,s)$ for all $t\ge s$ and    $t,s \in I$.

\section{Proof of main results}\label{s3}
\begin{proof}[Proof of Theorem \ref{t1}]
The first we show $e^{-\alpha |t|}\in E$. Put
$$ v(t)=\int_{-\infty}^{t} e^{-\alpha(t-\tau)} \chi_{[0,1]}(\tau) d\tau + \int_{t}^{\infty} e^{-\alpha(\tau-t)} \chi_{[0,1]}(\tau) d\tau.$$
Then,
$$ 
v(t)=\left \{ 
\begin{array}{llc}
\frac{e^{-\alpha t}(e^{\alpha} -1)}{\alpha}& \text{if }\; t\ge 1,&\\ 
\frac{e^{\alpha t}(1-e^{-\alpha} )}{\alpha}& \text{if }\; t\le 0,&\\
\frac{1-e^{-\alpha t} }{\alpha} + \frac{1-e^{-\alpha (1-t)} }{\alpha} & \text{if }\; t\in (0,1).&
\end{array}\right.
 $$
Therefore,
$e^{\alpha |t|} v(t) \ge \frac{1-e^{-\alpha} }{\alpha}$ for all $t\in \r$. On the other hand,
\begin{align*}
v(t)& = \sum_{k=0}^\infty \int_{t-(k+1)}^{t-k} e^{-\alpha(t-\tau)}\chi_{[0,1]}(\tau)d\tau +\sum_{k=0}^\infty \int_{t+k}^{t+k+1} e^{-\alpha(\tau-t)}\chi_{[0,1]}(\tau)d\tau\\
& \le \sum_{k=0}^\infty  e^{-\alpha k} \int_{t-(k+1)}^{t-k} \chi_{[0,1]}(\tau)d\tau +\sum_{k=0}^\infty e^{-\alpha k} \int_{t+k}^{t+k+1} \chi_{[0,1]}(\tau)d\tau\\
&= \sum_{k=0}^\infty  e^{-\alpha k} (T_{-k-1}(\Lambda_1\chi_{[0,1]}))(t) + \sum_{k=0}^\infty  e^{-\alpha k} (T_{k}(\Lambda_1\chi_{[0,1]}))(t)=:\p(t).
\end{align*}
This implies that
$$ e^{-\alpha|t|} \le \frac{\alpha}{1-e^{-\alpha}} v(t) \le \frac{\alpha}{1-e^{-\alpha}} \p(t) \quad \text{for all } \; t\in \r.$$
We also have the following estimates.
\begin{align*}
\sum_{k=0}^\infty  e^{-\alpha k} \| T_{-k-1} (\Lambda_1 \chi_{[0,1]}) \|_E + \sum_{k=0}^\infty  e^{-\alpha k} \| T_{k}(\Lambda_1 \chi_{[0,1]}) \|_E &\le 
\sum_{k=0}^\infty  e^{-\alpha k}\, 2N \| \Lambda_1 \chi_{[0,1]} \|_E  \\
&= \frac{2N}{1-e^{-\alpha}} \| \Lambda_1 \chi_{[0,1]} \|_E.  
\end{align*}
So, function series $\p$ is absolutely convergent in the Banach function space $E$. Therefore, $\p \in E$ and $\|\p\|_E \le \frac{2N \| \Lambda_1 \chi_{[0,1]} \|_E}{1-e^{-\alpha}}$.
By the item i. in Definition \ref{fs}, we get $e^{-\alpha|t|} \in E$ and
$$ \| e^{-\alpha|\cdot|} \|_E  \le \frac{2N \alpha\| \Lambda_1 \chi_{[0,1]} \|_E}{(1-e^{-\alpha})^2}.$$

Because $(U(t,s))_{t\ge s}$ is exponentially stable on $I$ so 
\begin{equation}\label{eq1}
 (\chi_{[t_0,\infty)}g_{t_0,x})(t) \le K_1 \chi_{[t_0,\infty)}(t) e^{-\alpha(t-t_0)}\|x\| \le K_1\|x\| (T_{-t_0}e^{-\alpha|\cdot|})(t)
\end{equation}
for all $t\in \r$. By the items i., v. in Definition \ref{fs} and \eqref{eq1}, we obtain $\chi_{[t_0,\infty)}g_{t_0,x} \in E$ and
$$ \|\chi_{[t_0,\infty)}g_{t_0,x}\|_E \le \frac{2N^2 K_1 \alpha\| \Lambda_1 \chi_{[0,1]} \|_E}{(1-e^{-\alpha})^2} \|x\| $$
for all $t_0\in I$.
\end{proof}

\begin{proof}[Proof of Theorem \ref{t2}]
For $t_0\in I$ and $t>t_0$, for each $x\in X$ then mapping $\p_x$ is defined as follows
$$ 
\p_x(\xi)=\begin{cases}
\|U(\xi,t_0)x\|& \text{ if }\; \xi \in [t_0,t],\\
0&\text{ if }\; \xi \notin [t_0,t].
\end{cases}
 $$
Then, $\p_x \in E$ and 
\begin{equation}\label{eq2}
\|\p_x\|_E \le \|\chi_{[t_0,\infty)}g_{t_0,x}\|_E \le M(x).
\end{equation}
We now construct a family of functions $\{\Phi_{t,t_0}\}$ determining by 
$$ \Phi_{t,t_0}: X \to \r \quad \text{with} \quad \Phi_{t,t_0}(x)=\|\p_x\|_E.$$
Easily see that $\Phi_{t,t_0}$ is a seminorm on $X$. On the other hand, we have
$$ \p_x(\xi) \le Ke^{c(t-t_0)}\chi_{[t_0,t]}(\xi) \|x\| $$
for all $\xi \in \r$. Thus, $\Phi_{t,t_0}(x) \le Ke^{c(t-t_0)} \|\chi_{[t_0,t]}\|_E \|x\|$. So, $\Phi_{t,t_0}$ is a continuous seminorm on $X$. Moreover, by \eqref{eq2} then
the family of continuous seminorms $\{ \Phi_{t,t_0}: t_0\in I, t>t_0 \}$ is pointwise bounded. Applying uniform boundedness principle (see Appendix), there exists a constant 
$C>0$ such that 
\begin{equation}\label{eq3}
 \Phi_{t,t_0}(x) \le C\|x\| 
\end{equation}
for all $x\in X$ and $t_0\in I, t>t_0$.

For $\xi \in [t_0,t]$, we have
$$ e^{-c(t-\xi)} \|U(t,t_0)x\| =   e^{-c(t-\xi)} \|U(t,\xi)U(\xi,t_0)x\| \le K\|U(\xi,t_0)x\|.$$
Therefore, 
$$ \chi_{[t_0,t]}(\xi)  e^{-c(t-\xi)} \|U(t,t_0)x\| \le K\p_x(\xi)$$
for all $\xi \in \r$. So, 
\begin{equation*}
\|\chi_{[t_0,t]}e^{-c(t-\cdot)}\|_E  \|U(t,t_0)x\| \le KC\|x\|.
\end{equation*}
By \eqref{ebm}, 
$$ \|\chi_{[t_0,t]}e^{-c(t-\cdot)}\|_M \|U(t,t_0)x\| \le \frac{KC M}{\inf_{\tau\in \r} \|\chi_{[\tau,\tau+1]}\|_E} \|x\| $$
for all $t_0 \in I$ and $t>t_0$.

For $t\ge t_0 +1$, we have
$$ \|\chi_{[t_0,t]}e^{-c(t-\cdot)}\|_M \ge \int_{t-1}^{t} e^{-c(t-\xi)} d\xi =\frac{1-e^{-c}}{c}.$$
Thus,
$$ \|U(t,t_0)x\| \le \frac{KC M c}{(1-e^{-c})\inf_{\tau\in \r} \|\chi_{[\tau,\tau+1]}\|_E} \|x\|  $$
for all $t_0\in I$ and $t\ge t_0+1$. Because the evolutionary family $(U(t,s))_{t\ge s}$ is exponentially bounded so there exists a constant $C_1>0$ such that
$$ \|U(t,t_0)x\| \le C_1 \|x\| $$
for all $x\in X$ and $t_0 \in I, t\ge t_0$.

For $\xi \in [t_0,t]$, we have $\|U(t,t_0)x\| \le C_1 \|U(\xi, t_0)x\|$. Therefore, $\chi_{[t_0,t]}(\xi) \|U(t,t_0)x\| \le C_1 \p_x(\xi)$ for all $\xi \in \r$. So,
$$ \|\chi_{[t_0,t]}\|_E  \|U(t,t_0)x\| \le C_1 C \|x\|$$
for all $t_0\in I$ and $t>t_0$. On the other hand,
$$ \chi_{[0,t-t_0]}(\xi) = \chi_{[t_0,t]}(\xi+t_0) =(T_{t_0}\chi_{[t_0,t]})(\xi).$$
Therefore, $\|\chi_{[0,t-t_0]}\|_E \le N  \|\chi_{[t_0,t]}\|_E$. So, 
$$ \|U(t,t_0)x\| \le \frac{NC_1C}{\|\chi_{[0,t-t_0]}\|_E} \|x\| $$
for all $t_0\in I $ and $ t>t_0$. 
Because of $\displaystyle\lim_{t\to+\infty}\|\chi_{[0,t]}\|_E=\infty$ so the evolutionary family $(U(t,s))_{t\ge s}$ is exponential stability on $I$.
\end{proof}

\begin{proof}[Proof of Corollary \ref{co1}]
For $t_0\in \r$ and $x\in X$, we have $g_{t_0,x}(t)=\|T(t-t_0)x\|$ for $t\ge t_0$. Therefore,
$$ (\chi_{[t_0,\infty)}g_{t_0,x})(t) = (\chi_{[0,\infty)}g_{0,x})(t-t_0)=(T_{-t_0}(\chi_{[0,\infty)}g_{0,x}))(t)$$
for $t\in \r$. By the item v. in Definition \ref{fs}, we get $\chi_{[t_0,\infty)}g_{t_0,x} \in E$ and 
$$ \|\chi_{[t_0,\infty)}g_{t_0,x}\|_E \le N \|\chi_{[0,\infty)}g_{0,x}\|_E \quad \text{for all } \; t_0 \in \r. $$
By Theorem \ref{t2}, $(T(t))_{t\ge 0}$ is an exponentially stable semigroup.
\end{proof}

\begin{proof}[Proof of Corollary \ref{co2}]
By the periodicity of the evolutionary family $(U(t,s))_{t\ge s}$ on $I$ so we just need  to consider $t_0\ge 0$.
We will repeat  manner of the proof in Theorem \ref{t2} to obtain \eqref {eq3} as follows.

The similar discussions, there exists a constant $D>0$ such that
$$ \Phi_{t, T}(x)\le D\|x\| $$
for all $x\in X$ and $t>T$. For $t_0 \in [0,T)$ and $t\in (t_0,T]$, we have $\p_x(\xi) \le \chi_{[0,T]}(\xi) e^{cT}\|x\|$ for all $\xi \in \r$. Therefore,
$\p_x \in E$ and $\|\p_x\|_E \le e^{cT} \|\chi_{[0,T]}\|_E \|x\|$. For $t>T$,
$$ \p_x(\xi) \le \chi_{[0,T]}(\xi) e^{cT} \|x\|+\chi_{[T,t]} (\xi) \|U(\xi,T)U(T,t_0)x\| $$
for all $\xi \in \r$. Thus,  $\p_x \in E$ and
$$ \|\p_x\|_E \le e^{cT} \|\chi_{[0,T]}\|_E  \|x\|+   \Phi_{t, T}(U(T,t_0)x) \le e^{cT}(\|\chi_{[0,T]}\|_E+D)\|x\|. $$

For $t_0\ge T$ and $t>t_0$, we can write $t_0=nT+\tau$ with $n\in \mathbb{N}$ and $\tau\in [0,T)$. Then,
$$ \|U(\xi,t_0)x\|=\|U(\xi, nT+\tau)x\|=\|U(\xi-nT,\tau)x\|=\|U(\xi-t_0+\tau,\tau)x\|$$
for $\xi \in [t_0,t]$. Put
$$ \psi_{x}(\xi)= 
\begin{cases}
\|U(\xi,\tau)x\|& \text{ if }\; \xi \in [\tau,t-t_0+\tau],\\
0&\text{ if }\; \xi \notin [\tau,t-t_0+\tau].
\end{cases}$$
Then, $\p_{x}(\xi)=(T_{-t_0+\tau}\psi_x)(\xi)$ for all $\xi \in \r$. Because of $\tau \in [0,T)$ so $\psi_x\in E$, hence  $\p_{x} \in E$  and 
$$ \|\p_x\|_E \le  N e^{cT}(\|\chi_{[0,T]}\|_E+D)\|x\|.$$
So, for all $t_0\ge 0$ and $t>t_0$ then $\p_x\in E$ and there is a constant $M(x)>0$ such that $\|\p_x\|_E \le M(x)$. Therefore, there exists a constant $C>0$ such that 
\begin{equation*}
 \Phi_{t,t_0}(x) \le C\|x\| 
\end{equation*}
for all $x\in X$ and $t_0\ge 0, t>t_0$. The next step, by the same discussions as in the proof of Theorem \ref{t2} we deduce that the evolutionary family $(U(t,s))_{t\ge s}$ is exponential stability on $I$.
\end{proof}

\subsection*{Appendix}
Let $X$ be a Banach space over the field $K$ ($K=\r$ or $\mathbb{C}$). A mapping $p: X\to \r_+$ is a seminorm on $X$ if $p(\theta x)=|\theta| p(x)$ and $p(x+y)\le p(x)+p(y)$ for all $x, y\in X$ and $\theta \in K$. As known that seminorm $p$ is continuous on $X$ if and only if $p$ is continuous at $0$. Moreover, if $A$ is a bounded linear operator on $X$ then $p_A(x)=\|Ax\|$ for $x\in X$ is a continuous seminorm on $X$.

Let $\Lambda$ be an index set. A family of continuous seminorms $\{ p_{\lambda} : \lambda \in \Lambda \}$ on $X$ is called
\begin{itemize}
 \item pointwise bounded if for each $x\in X$ there exists a constant $M(x)>0$ such that $p_{\lambda}(x) \le M(x)$ for all $\lambda \in \Lambda$;
\item uniformly bounded if there is a constant $M>0$ such that $p_{\lambda} (x) \le M \|x\|$ for all $\lambda \in \Lambda$ and $x\in X$.
\end{itemize}

The similar proof as the uniform boundedness principle for family of bounded linear operators (see Kreyszig \cite{Krey}),  we get version of 
uniform boundedness principle for family of  continuous seminorms.

\medskip
\noindent
{\bf Uniform boundedness principle.} {\it Let $\{ p_{\lambda} : \lambda \in \Lambda \}$ be a  family of continuous seminorms on $X$. If this family is pointwise bounded then it is also uniformly bounded.}

\medskip
Obviously, this version is more general than the old version in functional analysis.


\begin{thebibliography}{10} 

\bibitem{ChY}
C. Chicone, Y. Latushkin, 
{\it ``Evolution Semigroups in Dynamical Systems and Differential Equations''}, Amer. Math. Soc., Vol. \textbf{70}, 1999.

\bibitem{DaK}
Ju.L. Daleckii, M.G. Krein, {\it  ``Stability  of Solutions of Differential Equations in Banach Space''},
Transl. Math. Monogr. {\bf 43}, Amer. Math. Soc., Provindence, RI, 1974.

\bibitem{Dat}
R. Datko, {\it Uniform asymptotic stability of evolutionary processes in Banach spaces}, SIAM J. Math. Anal., \textbf{3} (1972), 428-445.

\bibitem{NaE}
K.J. Engel, R. Nagel, {\it ``One-parameter Semigroups for Linear Evolution Equations''},
Graduate Text Math., Vol. \textbf{194}, Springer-Verlag, Berlin-Heidelberg, 2000.

\bibitem{Krey}
E. Kreyszig, {\it ``Introductory Functional Analysis with Applications''}, John Wiley \& Sons. Inc., 1978.

\bibitem{MasSch}  J.L. Massera,  J.J. Sch\" affer, {\it ``Linear Differential Equations and Function Spaces''}, Academic Press, New York, 1966.

\bibitem{MinhRaSch}
N.V. Minh, F. R\"{a}biger, R. Schnaubelt,
 {\it Exponential stability, exponential expansiveness and exponential dichotomy of evolution equations on the half line}, 
Integral Equations Operator Theory {\bf32} (1998), 332-353.

\bibitem{VanNee}
J.M.A.M. van Neerven, {\it Exponential stability of operators and operator semigroups}, J. Funct. Anal., \textbf{130} (1995), 293-309 

\bibitem{Paz}
A. Pazy, {\it ``Semigroup of Linear Operators and Application to Partial Differential Equations''}, Springer-Verlag, Berlin, 1983.


\end{thebibliography}
\end{document}